[1994/12/01]
\documentclass{ijmart-mod}
\chardef\bslash=`\\ 

\usepackage{graphicx}
\usepackage[breaklinks=true]{hyperref}
\usepackage{hypcap}
\usepackage{mathtools}
\usepackage{xcolor}
\usepackage[ruled,linesnumbered]{algorithm2e}

\usepackage{multirow}
\usepackage{array}

\newtheorem{thm}{Theorem}[section]
\newtheorem{cor}[thm]{Corollary}
\newtheorem{lem}[thm]{Lemma}
\newtheorem{prop}[thm]{Proposition}

\theoremstyle{definition}

\newtheorem{rem}[thm]{Remark}

\theoremstyle{remark}

\newcommand{\eval}[2][\right]{\relax
  \ifx#1\right\relax \left.\fi#2#1\rvert}

\begin{document}
\title{Sizing the White Whale}

\author{Antoine Deza}
\address{McMaster University, Hamilton, Ontario, Canada}
\email{deza@mcmaster.ca}

\author{Mingfei Hao}
\address{McMaster University, Hamilton, Ontario, Canada}
\email{haom6@mcmaster.ca}

\author{Lionel Pournin}
\address{Universit{\'e} Paris 13, Villetaneuse, France}
\email{lionel.pournin@univ-paris13.fr}

\begin{abstract}
We propose a computational, convex hull free framework that takes advantage of the combinatorial structure of a zonotope, as for example its symmetry group, to orbitwise generate all canonical representatives of its vertices. We illustrate the proposed framework by generating all the 1\,955\,230\,985\,997\,140 vertices of the $9$\nobreakdash-dimensional \emph{White Whale}. We also compute the number of edges of this zonotope up to dimension $9$ and exhibit a family of vertices whose degree is exponential in the dimension. The White Whale is the Minkowski sum of all the $2^d-1$ non-zero $0/1$\nobreakdash-valued $d$\nobreakdash-dimensional vectors. The central hyperplane arrangement dual to the White Whale, made up of the hyperplanes normal to these vectors, is called the {\em resonance arrangement} and has been studied in various contexts including algebraic geometry, mathematical physics, economics, psychometrics, and representation theory. 
\end{abstract}
\maketitle

\section{Introduction}\label{CZ.sec.introduction}

Given a set $G=\{g^1,g^2,\ldots,g^m\}$ of non-zero vectors from $\mathbb{R}^d$, a zonotope $H_G$ can be defined as the convex hull of all the $2^m$ subsums of the vectors in $G$. 
Equivalently, $H_G$ is the Minkowski sum of the line segments between the origin of $\mathbb{R}^d$ and the extremity of a vector from $G$:
$$
H_G=\mathrm{conv}\left\{\sum_{j=1}^{m}\varepsilon_j g^j : \varepsilon_j\in\{0,1\}\right\}\!\mbox{.}
$$

Hereafter, the vectors contained in $G$ are referred to as the generators of $H_G$. The associated hyperplane arrangement $\mathcal{A}_{G}$ is made up of the hyperplanes
$$
H^j=\{x\in\mathbb{R}^d : x^Tg^j =0\}
$$
through the origin of $\mathbb{R}^d$ and orthogonal to a vector in $G$. The chambers, or regions, of $\mathcal{A}_{G}$ are the connected components of the complement in $\mathbb{R}^d$ of the union of the hyperplanes in $\mathcal{A}_{G}$. By the duality between zonotopes and hyperplane arrangements, the vertices of $H_G$ and the chambers of $\mathcal{A}_G$ are in one-to-one correspondence.

The characteristic polynomial $\chi(\mathcal{A}_{G};t)$ of $\mathcal{A}_{G}$ is defined as
$$\chi(\mathcal{A}_{G};t)= b_0(\mathcal{A}_{G})t^d-b_1(\mathcal{A}_{G})t^{d-1}+b_2(\mathcal{A}_{G})t^{d-2}\dots(-1)^d b_d(\mathcal{A}_{G}).$$
where the coefficients $b_i(\mathcal{A}_{G})$ are called the Betti numbers with $b_0(\mathcal{A}_{G})=1$ and $b_1(\mathcal{A}_{G})=m$~\cite{Stanley2012}. The number of chambers of $\mathcal{A}_{G}$, and thus the number of vertices of $H_G$, is equal to 
$b_0(\mathcal{A}_{G})+b_1(\mathcal{A}_{G})+\dots+b_d(\mathcal{A}_{G})$.\\

We propose a computational framework that goes beyond counting the vertices of $H_G$ as it explicitly generates all of these vertices. Since a zonotope is also a polytope, this can theoretically be achieved from a convex-hull computation. This kind of computation can be performed in a more efficient way by exploiting the potentially large symmetry group of $H_G$. Instead of generating all of the vertices of $H_G$, our framework restricts to generating one canonical representative in the orbit of each vertex under the action of that group. The whole vertex set of $H_G$ can then be recovered by letting the symmetry group of $H_G$ act on these representatives. Minkowski sum computations can be performed via recursive convex hulls by adding the generators one by one. We refer to~\cite{AvisBremnerSeidel1997,AvisFukuda1992,AvisJordan2018,DezaPournin2022,Fukuda2015,GawrilowJoswig2000}
and references therein for more details about  convex hull computations, orbitwise enumeration algorithms, and Minkowski sum computations. While a number of practical algorithms have been developed, this kind of task is highly computationally expensive. For this reason, our framework is convex hull free. It also exploits the combinatorial properties of Minkowski sums, and involves a linear optimization oracle whose complexity is polynomial in the number $m$ of generators. We establish additional combinatorial properties of a highly structured zonotope---the White Whale~\cite{Billera2019}---that allow for a significant reduction of the number of such linear optimization oracle calls, and thus to perform the orbitwise generation of all the 1 955 230 985 997 140 vertices of the $9$-dimensional White Whale. This zonotope appears in a number of contexts as for example algebraic geometry, mathematical physics, economics, psychometrics, and representation theory~\cite{Kuhne2021,ChromanSinghal2021,Evans1995,GutekunstMeszarosPetersen2019,KamiyaTakemuraTerao2011,Kuhne2020,vanEijck1995,Wang2013} and is a special case of the \emph{primitive zonotopes}, a family of zonotopes originally considered in relation with the question of how large the diameter of a lattice polytope can be \cite{DezaManoussakisOnn2018}. We refer to Fukuda~\cite{Fukuda2015}, Gr\"unbaum~\cite{Grunbaum2003}, and Ziegler~\cite{Ziegler1995} for polytopes and, in particular, zonotopes.

In Section~\ref{sec:zonotope}, we present two algorithms that exploit the combinatorial structure of a zonotope to compute its vertices. In Section~\ref{sec:whitewhale}, we give several additional properties of the White Whale that allows for an improved version of these algorithms, making it possible to orbitwise generate the vertices of the $9$\nobreakdash-dimensional White Whale. 
We then explain in Section~\ref{edge-gen} how the number of edges of the White Whale can be recovered from the list of its vertices, and provide these numbers up to dimension $9$. Finally, we study the degrees of its vertices in Section~\ref{sec:degree} and, in particular, we determine the degree in all dimensions of a particular family of vertices, which shows that the degree of some of the vertices of the White Whale is exponential in the dimension.

\section{Generating the vertices of a zonotope}\label{sec:zonotope}

By its combinatorial structure, linear optimization over a zonotope is polynomial in the number $m$ of its generators.  In particular, checking whether a point $p$, given as the sum of a subset $S$ of the generators of $H_G$, is a vertex of $H_G$ is equivalent to checking whether the following system of $m$ inequalities is feasible, which amounts to solving a linear optimization problem.
$$
(LO_{S,G})\left\{
\begin{array}{rcl}
c^Tg^j\geq1 & \mbox{ for all } & g^j\in S\mbox{,}\\
c^Tg^j\leq-1 & \mbox{ for all }  & g^j\in G\mathord{\setminus}S\mbox{.}
\end{array}
\right.
$$

Note that we can assume without loss of generality that no two generators of $H_G$ are collinear. In the sequel, we denote by $p(S)$ the sum of the vectors contained in a subset $S$ of $G$, with the convention that $p(\emptyset)$ is the origin of $\mathbb{R}^d$. Observe that for every vertex $v$ of $H_G$ there is a unique subset $S$ of $G$ such that $v$ is equal to $p(S)$. If $(LO_{S,G})$ is feasible; that is, if there exists a vector $c$ satisfying the above system of $m$ inequalities, then $p(S)$ is the unique point that maximizes $c^T x$ when $x$ ranges within $H_G$.

A brute-force linear optimization based approach would essentially consist in calling the oracle $(LO_{S,G})$ on each of the $2^m$ subsets $S$ of $G$. Since any edge of a zonotope is, up to translation, the line segment between the origin and an element of $G$, for any vertex $v=p(S)$ of $H_G$ with $S\neq\emptyset$  there exists a generator $g^i$ in $S$ such that $v$ and $p(S\mathord{\setminus}\{g^i\})$ are the vertices of an edge of $H_G$. Consequently, the brute-force approach can be enhanced by considering the following layered formulation, that results in Algorithm~\ref{LOG}. Consider the layer $\mathcal{L}_k(G)$ made up of the vertices of $H_G$ obtained as the sum of exactly $k$ of its generators. By a slight abuse of notation, we identify from now on a subset $S$ of $G$ such that $p(S)$ is a vertex of $H_G$ with the vertex itself. 
Recall that two different subsets of $G$ cannot sum to a same vertex of $H_G$. By this identification, $\mathcal{L}_k(G)$ can be written as follows:
$$
\mathcal{L}_k(G)=\{S\subseteq G \mbox{ such that } |S|=k \mbox{ and } p(S) \mbox{ is a vertex of } H_G \}\mbox{.}
$$

Assuming that $\mathcal{L}_k(G)$ is known, one can consider for each $S$ in $\mathcal{L}_k(G)$ the $m-k$ points $p(S)+g^j$ for $g^j\in G\backslash S$. Calling $(LO_{S,G})$ on all such points  $p(S)+g^j$ allows for the determination of all the vertices of $H_G$ that are equal to a subsum of exactly $k+1$ elements of $G$. 
That recursive layered approach allows for a significant speedup as the number of vertices equal to a subsum of exactly $k$ elements of $G$ is in practice much smaller that the upper bound of
$$
{m\choose{k}}
$$
and the number of $(LO_{S,G})$ calls is in practice much smaller than

$$
2^m=\sum_{k=0}^m{m\choose{k}}\!\mbox{.}
$$

In order to compute the layer $\mathcal{L}_{k+1}(G)$, one only needs knowledge of the previous layer $\mathcal{L}_k(G)$. In particular, the memory required by the algorithm is limited to the storage of only two consecutive layers. In Line 10 of Algorithm~\ref{LOG}, the layer $\mathcal{L}_{k+1}(G)$ that has just been computed is stored. At the same time, the layer  $\mathcal{L}_k(G)$ can be removed from the memory.

\begin{algorithm}[t]\label{LOG}
\KwIn{the set $G$ of all the $m=|G|$ generators of $H_G$}

$\mathcal{L}_0(G)\leftarrow \emptyset$

\For{$k=0,\dots,m-1$}{
\For{each $S\in\mathcal{L}_k(G)$}{
\For{each $g^j\in G\backslash S$}{
\If{$(LO_{S\cup \{ g^j\},G})$ is feasible}{
$\mathcal{L}_{k+1}(G)\leftarrow \mathcal{L}_{k+1}(G) \cup \{S\cup \{ g^j \}\}$
}
}
}
Save $\mathcal{L}_{k+1}(G)$
}
\caption{Layered optimization-based vertex generation}
\end{algorithm}
\begin{algorithm}[b]\label{LOOG}
\KwIn{set $G$ of all the $m=|G|$ generators of $H_G$}

 $\widetilde{\mathcal{L}}_0(G)\leftarrow\emptyset$

\For{$k=0,\dots,\lfloor m/2 \rfloor-1$}{

$i\leftarrow0$

\For{each $S\in\widetilde{\mathcal{L}}_k(G)$}{
\For{each $g^j\in G\backslash S$}{
\If{$(O_{S\cup\{ g^j\},G})$ returns {\sc true}}{
\If{$(LO_{S\cup \{ g^j\},G})$ is feasible}{

$S_{k+1}^i\leftarrow${\em canonical representative of} $S\cup \{ g^j \}$

\If{$S_{k+1}^i$ does not belong to $\widetilde{\mathcal{L}}_{k+1}(G)$}{

$\widetilde{\mathcal{L}}_{k+1}(G)\leftarrow\widetilde{\mathcal{L}}_{k+1}(G)\cup \{S_{k+1}^i\}$   

$i\leftarrow{i+1}$
}
}
}
}
}
Save $\widetilde{\mathcal{L}}_{k+1}(G)$
}
\caption{Layered optimization-based orbitwise vertex generation}
\end{algorithm}

It should be noted that Algorithm~\ref{LOG} is a layered version of an algorithm given in \cite{DezaPournin2022}. It can be significantly improved into Algorithm~\ref{LOOG} by exploiting the structural properties of a zonotope $H_G$ as follows.


\begin{rem} Consider a zonotope $H_G$ with $m=|G|$ generators.
\begin{itemize}
\item[$(i)$] $H_G$ is centrally symmetric with respect to the point
$$
\sigma=\frac{1}{2}p(G)\mbox{.}
$$
The point $p(S)$ is a vertex of $H_G$ if and only if $p(G\backslash S)$ is a vertex of $H_G$. Thus, when considering an orbitwise generation of the vertices of $H_G$, we can assume without loss of generality that $|S|\leq \lfloor m/2 \rfloor$.
\item[$(ii)$] Assuming that $G$ is invariant under the action of a linear transformation group, as for example coordinate permutations, an orbitwise generation can be performed by replacing $\mathcal{L}_k(G)$ with the set $\widetilde{\mathcal{L}}_k(G)$ of all canonical representatives of the points from $\mathcal{L}_k(G)$. For coordinate permutations, $\widetilde{\mathcal{L}}_k(G)$ is the set of all the vertices of $\mathcal{L}_k(G)$ such that
$$
p_i(S)\leq p_{i+1}(S)
$$
for all integers $i$ satisfying $1\leq{i}<d$.
\item[$(iii)$] Assuming that an oracle $(O_{S,G})$ certifying that $p(S)$ is not a vertex is available and computationally more efficient than $(LO_{S,G})$, we can further speed the algorithm up by calling $(O_{S,G})$ before calling $(LO_{S,G})$. Typically, $(O_{S,G})$ is a heuristic that returns {\sc false}  if  $(O_{S,G})$ is able to show that theres exists a subset $T$ of $G$ distinct from $S$ such that $p(S)=p(T)$.  Thus, $p(S)$ admits two distinct decompositions into a subsum of $G$ and therefore, it cannot be a vertex of $H_G$. If that oracle is able to detect most of the subsums of generators of $H_G$ that do not form a vertex of $H_G$, this results in a significant speedup.
\end{itemize}
\end{rem}

Observe that, in Line 7 of Algorithm~\ref{LOOG}, the subset $S^i_{k+1}$ of $G$ added into $\widetilde{\mathcal{L}}_{k+1}(G)$, should be the one such that $p(S^i_{k+1})$ is 
the canonical  representative in the orbit of $p(S\cup\{g^j\})$ under the action of the chosen group. As was the case with Algorithm~\ref{LOG}, only two consecutive layers need to be kept in the memory by Algorithm~\ref{LOOG}. For instance, layer $\widetilde{\mathcal{L}}_k(G)$ can be deleted from memory in Line 17. As we shall see in Section~\ref{edge-gen}, that layered optimization-based vertex generation of $H_G$ also allows for the determination of all the edges of $H_G$.

\section{Generating the vertices of the White Whale}\label{sec:whitewhale}

We first recall a few results concerning the White Whale. Using the notations of~\cite{DezaManoussakisOnn2018,DezaPourninRakotonarivo2021}, the White Whale is the primitive zonotope $H_{\infty}^+(d,1)$ defined as the Minkowski sum of the $2^d-1$ non-zero $0/1$-valued $d$-dimensional vectors. Let us denote by $a(d)$ the number of vertices of $H_{\infty}^+(d,1)$.  For example $H_{\infty}^+(3,1)$ is the zonotope with $a(3)=32$ vertices shown in Figure~\ref{Fig_H3}. Its seven generators are the vectors $(1,0,0)$, $(0,1,0)$, $(0,0,1)$, $(0,1,1)$, $(1,0,1),(1,1,0)$, and $(1,1,1)$. The central arrangement associated to $H_{\infty}^+(d,1)$, the $d$-dimensional resonance arrangement is denoted by $\mathcal{R}_d$, see~\cite{GutekunstMeszarosPetersen2019} and references therein.

\begin{figure}[b]
\begin{centering}
\includegraphics[scale=1]{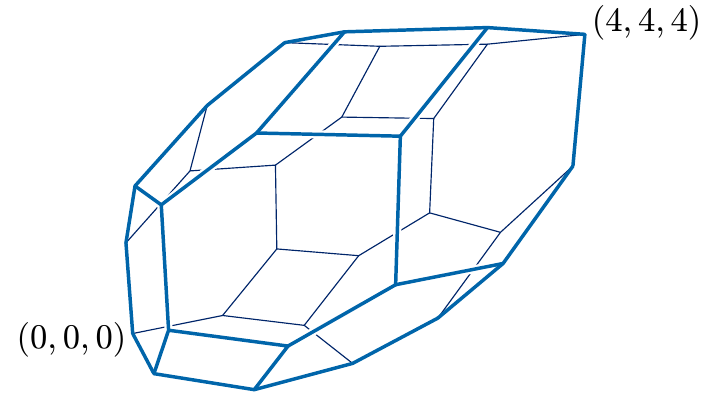}
\caption{The 3-dimensional White Whale $H_{\infty}^+(3,1)$.}\label{Fig_H3}
\end{centering}
\end{figure}

\begin{prop}\label{combi}
The White Whale  $H_{\infty}^+(d,1)$ has the following properties.
\begin{itemize}
\item[$(i)$] $H_{\infty}^+(d,1)$ is invariant under the symmetries of $\mathbb{R}^d$ that consist in permuting coordinates, see \cite{DezaManoussakisOnn2018}. 
\item[$(ii)$] $H_{\infty}^+(d,1)$ is contained in the hypercube $ [0,2^{d-1}]^d$ and the intersection of $H_{\infty}^+(d,1)$ with any facet of that hypercube coincides, up to translation and rotation with $H_{\infty}^+(d-1,1)$, see \cite{DezaManoussakisOnn2018}.
\item[$(iii)$] The number of vertices $a(d)$ of $H_{\infty}^+(d,1)$ is an even multiple of $d+1$, and satisfies (see \cite{DezaPourninRakotonarivo2021,GutekunstMeszarosPetersen2019,Wang2013})
$$
\frac{d+1}{2^{d+1}}2^{d^2(1-10/\ln d)}\leq a(d)\leq  \frac{d+4}{2^{3(d-1)}}2^{d^2}\mbox{.}
$$
\end{itemize}
\end{prop}

In view of assertion $(i)$ in the statement of Proposition~\ref{combi}, we call a vertex $v$ of $H_{\infty}^+(d,1)$ \emph{canonical} when $v_i\leq v_{i+1}$ for $1\leq{i}\leq{d-1}$. The values of $a(d)$ have been determined up to $d=9$ as recorded in sequence A034997 of the On-Line Encyclopedia of Integer Sequences~\cite{OEIS}. We report these values in Table~\ref{Table_a(d)} along with the references where they are obtained. The authors of the references where $a(d)$ is determined via the characteristic polynomial of $\mathcal{A}_{G}$; that is by counting, are indicated using  {\sc capital letters}.

\begin{rem}
By Proposition~\ref{combi}, $a(d)$ is even and a multiple of $d+1$. Interestingly, when $d$ is equal to $5$, we obtain from Table~\ref{Table_a(d)} that
$$
\frac{a(d)}{2(d+1)}=941\mbox{,}
$$
which is a prime number.
\end{rem}

If the aim is to count but not to generate the vertices of $H_{\infty}^+(d,1)$, the approach proposed by Kamiya, Takemura, and Terao~\cite{KamiyaTakemuraTerao2011} can be applied. It was enhanced by Chroman and Singhal \cite{ChromanSinghal2021} who determined the characteristic polynomial of the $9$-dimensional resonance arrangement $\mathcal{R}_9$. In addition, a formula for Betti numbers $b_2(\mathcal{R}_d)$ and $b_3(\mathcal{R}_d)$ has been given by K{\"u}hne~\cite{Kuhne2020}, and a formula for $b_4(\mathcal{R}_d)$ by Chroman and Singhal~\cite{ChromanSinghal2021}.  Pursuing the characteristic polynomial approach, Brysiewicz, Eble, and K{\"u}hne~\cite{Kuhne2021} computed the Betti numbers for a number of  hyperplane arrangements with large symmetry groups and, independently and concurrently confirmed the value of $a(9)$.

\begin{table}[t]
$$
\begin{array}{c|c|c}
d & a(d)  &  \mbox{References} \\
\hline
2 & 6 & \mbox{{Evans}~\cite{Evans1995} (1995)} \\
3 & 32 & \mbox{{Evans}~\cite{Evans1995} (1995)} \\
4 & 370 &  \mbox{{Evans}~\cite{Evans1995} (1995), {van Eijck}~\cite{vanEijck1995} (1995)} \\
5 & 11\,292 &  \mbox{{Evans}~\cite{Evans1995} (1995), {van Eijck}~\cite{vanEijck1995} (1995)} \\
6 & 1\,066\,044 &  \mbox{{Evans}~\cite{Evans1995} (1995), {van Eijck}~\cite{vanEijck1995} (1995)} \\
7 & 347\,326\,352 & \mbox{{van Eijck}~\cite{vanEijck1995} (1995), {\sc Kamiya et al.}~\cite{KamiyaTakemuraTerao2011} (2011)} \\
8 & 419\,172\,756\,930 &  \mbox{{Evans}~\cite{OEIS} (2011)} \\
9 & 1\,955\,230\,985\,997\,140 &  \mbox{{\sc Brysiewicz, Eble, and K{\"u}hne}~\cite{Kuhne2021} (2021)},\\ 
							& &  \mbox{{\sc Chroman and Singhal}~\cite{ChromanSinghal2021} (2021)} \\
\end{array}
$$
\caption{Generating and {\sc counting} the vertices of the White Whale.}\label{Table_a(d)}
\end{table}

From now on, we denote by $G_d$ the set of the $2^d-1$ generators of $H_{\infty}^+(d,1)$. Throughout the article, we will illustrate the proposed methods using the following family of vertices. When $1\leq{k}\leq{d-1}$, denote by $U_d^k$ the set of all the $0/1$-valued $d$-dimensional vectors whose last coordinate is equal to $1$ and that admit at most $k$ non-zero coordinates. For example, when $k=2$,
$$
U_d^2=
\left\{
\left[
\begin{array}{c}
1\\
0\\
0\\
\vdots\\
0\\
1\\
\end{array}\right]\!\mbox{, }
\left[
\begin{array}{c}
0\\
1\\
0\\
\vdots\\
0\\
1\\
\end{array}\right]\!\mbox{, }\ldots\mbox{, }
\left[
\begin{array}{c}
0\\
0\\
\vdots\\
0\\
1\\
1\\
\end{array}\right]\!\mbox{, }
\left[
\begin{array}{c}
0\\
0\\
0\\
\vdots\\
0\\
1\\
\end{array}\right]
\right\}\!\mbox{,}
$$
and $p(U_d^2)$ is equal to $(1,\dots,1,d)$. In general,
$$
p(U_d^k)=\left(\sum_{i=0}^{k-2}{d-2 \choose i},\dots,\sum_{i=0}^{k-2}{d-2 \choose i},\sum_{i=0}^{k-1}{d-1 \choose i}\right)\mbox{.}
$$

Proposition~\ref{sommet} illustrates how $(LO_{S,G_d})$ can be used to identify the vertices of the White Whale in any dimension in the special case of $p(U_d^k)$. 

\begin{prop}\label{sommet}
The point $p(U_d^k)$ is a canonical vertex of $H_\infty^+(d,1)$.
\end{prop}

\begin{proof}
As the coordinates of $p(U_d^k)$ are nondecreasing, if this point is a vertex of $H_\infty^+(d,1)$, it must be canonical. We consider  the $d$-dimensional vector
$$
c=(-2,\dots,-2,2k-1)
$$
and use $(LO_{S,G_d})$ with $S=U_d^k)$ to show that $p(U_d^k)$ is indeed a vertex of $H_\infty^+(d,1)$.
%
If $g$ is a vector in $U_d^k$, then $c^Tg\geq1$. Now if $g$ belongs to $G_d\mathord{\setminus}U_d^k$, then either $g_d=0$ or at least $k$ of its $d-1$ first coordinates are non-zero. In the former case, $c^Tg\leq-2$ because $g$ has at least one non-zero coordinate.

In the latter case,
$$
c_1g_1+\dots+c_{d-1}g_{d-1}\leq-2
$$
and $c_dg_d=2k-1$. Hence $c^Tg\leq-1$ and the result follows.
\end{proof}

Observe that the last coordinate of $p(U_d^k)$ is precisely the number $l$ of elements of $U_d^k$ and thus $p(U_d^k)$ belongs to $\widetilde{\mathcal{L}}_l(G)$. Using a similar approach as in Proposition~\ref{sommet}, one can obtain other families of canonical vertices of the White Whale. For instance, according to Proposition~\ref{sommets}, the sum of the generators belonging to the subset $W_d^k$ of $G_d$ made up of the $2^k-1$ vectors whose first $d-k$ coordinates are equal to zero is a vertex of $H_\infty^+(d,1)$.

\begin{prop}\label{sommets}
$\:$
\begin{itemize}
\item[$(i)$] 
The point $p(W_d^k)=(0,\dots,0,2^{k-1},\dots,2^{k-1})$ whose first $d-k$ coordinates are equal to $0$ and whose last $k$ coordinates are equal to $2^{k-1}$ is a canonical vertex of $H_\infty^+(d,1)$ that belongs to $\widetilde{\mathcal{L}}_{2^{k}-1}(G_d)$.
\item[$(ii)$]
The only non-zero $0/1$-valued canonical vertex of $H_\infty^+(d,1)$  is $(0,\dots,0,1)$ and therefore, $\widetilde{\mathcal{L}}_{1}(G_d)=\{(0,\dots,0,1)\}$.
\end{itemize}
\end{prop}
\begin{proof}
In order to prove assertion $(i)$, consider the vector $c$ whose first $d-k$ coordinates are equal to $0$ and whose last $k$ coordinates are $1$. It suffices so use $(LO_{S,G_d})$ with $S=W_d^k$ to show that $p(W_d^k)$ is a vertex of $H_\infty^+(d,1)$. As the coordinates of this point are nondecreasing, it is a canonical vertex of $H_\infty^+(d,1)$. Observing that there are exactly $2^{k}-1$ vectors $g$ in $G_d$ such that $c\mathord{\cdot}g>0$ further shows that this vertex belongs to $\widetilde{\mathcal{L}}_{2^{k}-1}(G_d)$.

Observe that taking $k=1$ in assertion $(i)$ proves that $(0,\dots,0,1)$ is a canonical vertex of $H_\infty^+(d,1)$. In order to prove assertion $(ii)$ recall that a vertex of $H_\infty^+(d,1)$ is the sum of a unique subset of $G_d$. However, any point from $\{0,1\}^d$ with at least two non-zero coordinates can be written as the sum of several different subsets of $G_d$ (as for instance the subset that contains the point itself, and a subset that contains several points with only one non-zero coordinate). 
\end{proof}

Lemmas \ref{111} to \ref{edge} below, where ${\bf 1}$ denotes the generator $(1,\dots,1)$, are building blocks for an oracle that efficiently identifies that $p(S)$ is not a vertex of $H_{\infty}^+(d,1)$ for most subsets $S$ of $G_d$, by providing a computationally easy to check necessary condition for being a vertex of $H_{\infty}^+(d,1)$.
\begin{lem}\label{111}
Consider a subset $S$ of $G_d$ such that $p(S)$ is a vertex of $H_{\infty}^+(d,1)$. The vector ${\bf 1}$ belongs to $S$ if and only if $|S|\geq 2^{d-1}$.
\begin{proof}
The $2^d-2$ vectors in $G_d\backslash\{{\bf 1}\}$ can be partitioned into $2^{d-1}-1$ unordered pairs $\{g^i,\bar{g}^i\}$ such that $g^i+\bar{g}^i={\bf 1}$. Assume that ${\bf 1}$ belongs to $S$ and that, for some $i$, neither of the vectors in the pair $\{g^i,\bar{g}^i\}$ belong to $S$, then
$$
p(S)= p([S\mathord{\setminus}\{{\bf 1}\}]\cup\{ g^i, \bar{g}^i\})\mbox{.}
$$

Therefore, $p(S)$ admits two distinct decompositions, and thus can not be a vertex. It follows that, in addition to ${\bf 1}$, $S$ contains at least $2^{d-1}-1$ generators; that is $|S|\geq 2^{d-1}$.  Since $p(S)$ is a vertex of $H_{\infty}^+(d,1)$ if and only if $p(G_d\mathord{\setminus}S)$ is a vertex of $H_{\infty}^+(d,1)$, ${\bf 1}\in S$ if and only if $|S|\geq 2^{d-1}$.
\end{proof}
\end{lem}

\begin{lem}\label{edge111}
Any edge of the $d$-dimensional White Whale that coincides, up to translation, with the line segment between the origin of $\mathbb{R}^d$ and the point ${\bf 1}$ connects a vertex that is the sum of exactly $2^{d-1}-1$ generators to a vertex that is the sum of exactly $2^{d-1}$ generators.
\begin{proof}
This is a direct consequence of Lemma~\ref{111}.
\end{proof}
\end{lem}

When $k=d-1$, assertion $(i)$ of Proposition~\ref{sommets} tells that the point
$$
p(W_d^{d-1})=(0,2^{d-2},\dots,2^{d-2})
$$
is a canonical vertex that belongs to $\mathcal{L}_{2^{d-1}-1}(G_d)$, which provides an illustration of Lemma~\ref{edge111} with the edge of $H_\infty^+(d,1)$ whose endpoints are $p(W_d^{d-1})$ and $p(W_d^{d-1}\cup\{ {\bf 1}\})$. For example, when $d=3$, the segment with vertices $(0,2,2)$ and $(1,3,3)$ is an edge of the $H_\infty^+(3,1)$ as shown in Figure~\ref{Fig_H3-L}.

\begin{lem}\label{barg}
Consider a subset $S$ of $G_d$ such that $p(S)$ is a vertex of $H_{\infty}^+(d,1)$ and a vector $g^j$ in $S$. If $|S|< 2^{d-1}$, then ${\bf 1}-g^j$ does not belong to $S$.
\begin{proof}
Assume that $|S|< 2^{d-1}$. By Lemma~\ref{111}, $S$ cannot contain ${\bf 1}$. Assume that both  $g^j$ and ${\bf 1}-g^j$ belong to $S$. In this case, 
$$
p(S)= p([S\backslash \{g^j,{\bf 1}-g^j\}]\cup\{{\bf 1}\})\mbox{}
$$
and $p(S)$ would admit two distinct decompositions, a contradiction.
\end{proof}
\end{lem}

Proposition~\ref{sommet}, Lemma~\ref{sommets}, and Lemma~\ref{edge111} are illustrated in Figure~\ref{Fig_H3-L} where the vertices of $H_\infty^+(d,1)$ contained in the layer $\mathcal{L}_{k}(G_d)$ are marked $\circ$ when $k$ is even and $\bullet$ when $k$ is odd. The marks of the canonical vertices of $H_\infty^+(d,1)$ are further circled, and the edges equal, up to translation, to the line segment whose endpoints are  the origin of $\mathbb{R}^d$ and the point ${\bf 1}$ are colored red.

\begin{figure}[t]
\begin{centering}
\includegraphics[scale=1]{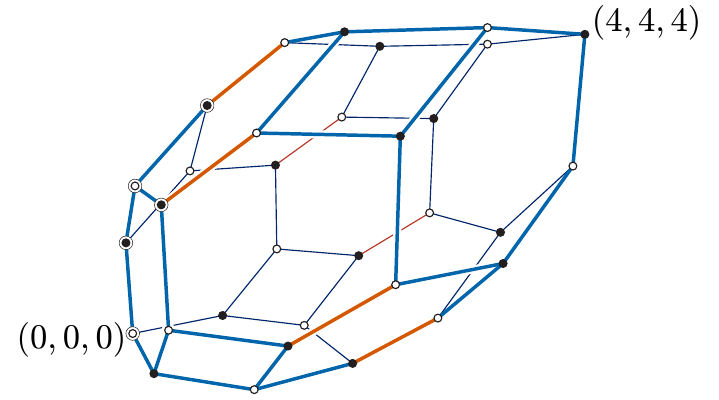}
\caption{The partition into eight layers of the vertex set of the $3$-dimensional White Whale $H_{\infty}^+(3,1)$.}\label{Fig_H3-L}
\end{centering}
\end{figure}

For a generator $g^j\in G_d$, let $\sigma(g^j)$ denote the {\em support} of $g^j$; that is the number of coordinates of $g^j$ that are equal to $1$. For any subset $S$ of $G_d$ and any vector $g^j$ in $G_d$, consider the following subset of $S$:
$$
S\langle g^j\rangle=\{g\in{S}:\mbox{ if }g_i^j=0\mbox{, then }g_i=0\mbox{ for }  1\leq{i}\leq{d} \}\mbox{,}
$$
or equivalently
$$
S\langle g^j\rangle=\{g\in{S}:g_i\wedge {g_i^j}=g_i\mbox{ for }  1\leq{i}\leq{d} \}\mbox{.}
$$

Lemma~\ref{edge} is a generalization of Lemma~\ref{edge111} that provides an easy to check necessary condition to be applied before calling $(LO_{S,G_d}$).

\begin{lem}\label{edge}
Consider a subset $S$ of $G_d$ such that $p(S)$ is a vertex of $H_{\infty}^+(d,1)$ and a vector $g^j$ contained in $G_d\mathord{\setminus}S$. If $|S\langle{g^j}\rangle|$ is not equal to $2^{\sigma(g^j)-1}-1$ then $p(S\cup\{g^j\})$ is not a vertex of $H_{\infty}^+(d,1)$.
 \begin{proof}
 The $2^{\sigma(g^j)}-2$ vectors in $G_d\langle g^j\rangle\mathord{\setminus}\{ g^j \}$ can be partitioned into $2^{\sigma(g^j)-1}-1$ unordered pairs  $\{g^l,\bar{g}^l\}$ such that $g^l+\bar{g}^l=g^j$.
If, for some $l$, neither of the vectors in the pair $\{g^l,\bar{g}^l\}$ belong to $S\langle{g^j}\rangle$, then
$$
p(S\cup\{g^j\})= p(S\cup\{ g^l,\bar{g}^l\})\mbox{.}
$$

In other words, $p(S\cup\{g^j\})$ can be obtained as the sums of two different subsets of $G_d$ and, therefore it cannot be a vertex of $H_\infty^+(d,1)$.

Now assume that, for some $l$, both $g^l$ and $\bar{g}^l$ belong to $S\langle{g^j}\rangle$. Then
$$
p(S)= p([S\mathord{\setminus}\{ g^l,\bar{g}^l\}]\cup\{g^{j}\})\mbox{.}
$$

It follows that $p(S)$ is obtained as the sums of two different subsets of $G_d$ and cannot be a vertex of $H_\infty^+(d,1)$, a contradiction.

This shows that, in order for $p(S\cup\{g^j\})$ to be a vertex of $H_\infty^+(d,1)$, it is necessary that $S\langle{g^j}\rangle$ contains exactly one vector from each of the $2^{\sigma(g^j)-1}-1$ unordered pairs  $\{g^l,\bar{g}^l\}$ of vectors such that $g^l+\bar{g}^l=g^j$, as desired.
 \end{proof}
\end{lem}

Lemma~\ref{edge} immediately results in an oracle $(O_{S\cup\{g^j\},G_d})$, that returns {\sc false} when $S\langle{g^j}\rangle$ does not contain exactly $2^{\sigma (g^j)-1}-1$ vectors; that is, when the point $p(S\cup \{g^j\})$ is certified not to be a vertex of $H_\infty^+(d,1)$. Computationally, calling $(O_{S\cup\{g^j\},G_d})$ first is significantly more efficient than just calling $(LO_{S\cup\{g^j\},G_d})$ because, in practice it allows to quickly discard a large number of candidates for vertexhood. Proposition~\ref{L2} illustrates how $(O_{S\cup\{g^j\},G_d})$ can be used to identify vertices of the White Whale in any dimension.

\begin{prop}\label{L2}
For any $d\geq 2$, $\widetilde{\mathcal{L}}_{2}(G_d)$ is equal to $\{(0,\dots,0,1,2)\}$, or equivalently
to $\{S_2^1\}$ where $S_2^1=\{(0,\dots,0,1),(0,\dots,0,1,1)\}$.
\begin{proof}
Consider a vertex $p(S)$ in $\widetilde{\mathcal{L}}_{k}(G_d)$ and a vector $g^j$ in $G\backslash S$. Since $S\langle{g^j}\rangle$ is a subset of $S$ and $g^j$ does not belong to $S$, the condition that $S\langle{g^j}\rangle\cup\{g^j\}$ admits exactly $2^{\sigma(g^j)-1}$ elements implies 
$$
2^{\sigma(g^j)-1}\leq |S|+1\mbox{.}
$$

As in addition, $S$ contains exactly $k$ elements, 
$$
{\sigma(g^j)}\leq 1+\lfloor\log_2(k+1)\rfloor\mbox{.}
$$

Hence, taking $k=1$ yields ${\sigma(g^j)}\leq 2$. By assertion $(ii)$ in the statement of Proposition~\ref{sommets}, $\widetilde{\mathcal{L}}_{1}(G_d)=\{(0,\dots,0,1)\}$ and no other $0/1$-valued point is a vertex of $H_{\infty}^+(d,1)$ . Consequently, $g^j$ must satisfy $g^j_d=1$. Since ${\sigma(g^j)}\leq 2$, the only possible candidate for $g^j$ is, up to the relabeling of the first $d-1$ coordinates, the vector $(0,\dots,0,1,1)$. 
Since  $(LO_{S,G_d})$ is feasible for $d=2$ and
$$
S=\{(0,\dots,0,1),(0,\dots,0,1,1)\}\mbox{,}
$$
we obtain $\widetilde{\mathcal{L}}_{2}(G_d)=\{(0,\dots,0,1,2)\}$ as desired.
\end{proof}
\end{prop}

Using a similar approach as in Proposition~\ref{L2}, one obtains the first few canonical vertex layers of the White Whale. We recall that $S^i_k$ denotes the $i^{th}$ canonical vertex of the layer $\widetilde{\mathcal{L}}_{k}(G_d)$.

\begin{prop}\label{Lk}
The following assertions hold.
\begin{itemize}
\item[$(i)$]
For any $d\geq 3$, $\widetilde{\mathcal{L}}_{3}(G_d)$ is equal to $\{(0,\dots,0,2,2),(0,\dots,0,1,1,3)\}$, or equivalently to $\{S^1_3,S^2_3\}$ where
$$
\left\{
\begin{array}{l}
S^1_3=S^1_2\cup \{(0,\dots,0,0,1,0)\}\mbox{,}\\
S^2_3=S^1_2\cup\{(0,\dots,0,1,0,1)\}\mbox{.}\\
\end{array}
\right.
$$

\item[$(ii)$]
For any $d\geq 4$, $\widetilde{\mathcal{L}}_{4}(G_d)$ is equal to
$$
\{(0,\dots,0,1,3,3),(0,\dots,0,2,2,4),(0,\dots,0,1,1,1,4)\}\mbox{,}
$$
or equivalently to $\{S^1_4,S^2_4,S^3_4\}$ where 
$$
\left\{
\begin{array}{l}
S^1_4=S^1_3\cup\{(0,\dots,0,0,1,1,1)\}\mbox{,}\\
S^2_4=S^2_3\cup\{(0,\dots,0,0,1,1,1)\}\mbox{,}\\
S^3_4=S^2_3\cup\{(0,\dots,0,1,0,0,1)\}\mbox{.}\\
\end{array}
\right.
$$
\end{itemize}
\end{prop}

Lemma~\ref{edge} allows to exploit the structure of the {White Whale in order to further enhance Algorithm~\ref{LOOG}, resulting in Algorithm~\ref{LOOGd} that can be used to efficiently generate all the canonical vertices of the White Whale.

\begin{algorithm}[b]\label{LOOGd}
\KwIn{the dimension $d$}
 $\widetilde{\mathcal{L}}_0(G)\leftarrow\emptyset$

\For{$k=0,\dots,2^{d-1}-2$}{

$i\leftarrow0$

\For{each $S\in\widetilde{\mathcal{L}}_k(G_d)$}{
\For{each $g^j\in G_d\backslash S$}{
\If{$(O_{S\cup \{g^j\},G_d})$ returns {\sc true}}{
\If{$(LO_{S\cup \{ g^j\},G_d})$ is feasible}{
$S^i_{k+1}\leftarrow$ {\em canonical representative of} $S\cup \{ g^j \}$

\If{$S^i_{k+1}$ does not belong to $\widetilde{\mathcal{L}}_{k+1}(G)$}{

$\widetilde{\mathcal{L}}_{k+1}(G)\leftarrow\widetilde{\mathcal{L}}_{k+1}(G)\cup \{S^i_{k+1}\}$

$i\leftarrow{i+1}$

}
}
}
}
}
Save $\widetilde{\mathcal{L}}_{k+1}(G_d)$
}
\caption{Orbitwise vertex generation for the White Whale}
\end{algorithm}


Note that in Line 5 of Algorithm~\ref{LOOGd}, we can restrict to only consider the vectors $g^j$ in $G_d\mathord{\setminus}S$ distinct from ${\bf 1}$ (by Lemma~\ref{111}), such that ${\bf 1}-g^j$ does not belong to $S$ (by Lemma~\ref{barg}), and such that $g^j_i\leq g^j_{i+1}$ when $p(S)_i=p(S)_{i+1}$ (by the assertion $(i)$ from Proposition~\ref{combi}).

We benchmarked Algorithm~\ref{LOOGd} by generating all the canonical vertices of $H_{\infty}^+(d,1)$ till $d=9$. As an illustration, we list all the points in $\widetilde{\mathcal{L}}_{k}(G_d)$ for $0\leq{k}\leq2^{d-1}-1$ when $d=3$ in Table~\ref{a3-vertices} and when $d=4$ in Table~\ref{a4-vertices}, where $|\mathcal{O}_{p(S)}|$ denotes the size of the orbit generated by the action of the symmetry group of $H_{\infty}^+(d,1)$ on a canonical vertex $p(S)$. 

There are different implementations of the algorithm based on the size of the solution space. For $d=1,\dots,8$, the algorithm is directly executed on a \texttt{CPython} interpreter, which is optimized through \texttt{Cython} and accelerated by the \texttt{IBM CPLEX} optimizer. Although layers are calculated sequentially due to their geometrical positions, the vertex candidates are partitioned into bundles and dispatched to multiple processes for further CPU-bound calculations.

For $d=9$, the algorithm is implemented as an \texttt{Apache Spark} pipeline. The task distribution, result collection and deduplication are managed by the underlying computation engine while the vertex-checking oracles are programmed as a map-reduce step, which is a \texttt{Python} script scheduled by \texttt{Spark} executors. The computation was run on an Ubuntu 16.04 server with a total of 72 threads $2\times$Intel\textsuperscript{\tiny\textregistered} Xeon\textsuperscript{\tiny\textregistered} Processor E5-2695 v4) and 300GB memory, and required 3 months of computational time. The output is stored on a cloud storage.
\begin{table}[b]
$$
\renewcommand{\arraystretch}{1.2}
\begin{array}{c|c|c|c}
\widetilde{\mathcal{L}}_k(G_3) & S^i_k & p(S^i_k) & |\mathcal{O}_{p(S^i_k)}|\\[0.5\smallskipamount]
\hline
\hline
\widetilde{\mathcal{L}}_0(G_3) & S^1_0=\emptyset & (0,0,0) & 2\\
\hline
\widetilde{\mathcal{L}}_1(G_3) & S^1_1=S^1_0\cup\{(0,0,1)\} & (0,0,1) & 6\\
\hline
\widetilde{\mathcal{L}}_2(G_3) & S^1_2=S^1_1\cup\{(0,1,1)\} & (0,1,2) & 12\\
\hline
\widetilde{\mathcal{L}}_3(G_3) & S^1_3=S^1_2\cup\{(0,1,0)\} & (0,2,2) & 6\\[-\smallskipamount]  
& S^2_3=S^1_2\cup\{(1,0,1)\} & (1,1,3) & 6\\
\hline
\hline
& & & a(3)=\sum |\mathcal{O}_{p(S^i_k)}|=32\\
\end{array}
$$
\caption{Sizing the $3$-dimensional  White Whale}\label{a3-vertices}
\end{table}

\begin{table}[t]
$$
\renewcommand{\arraystretch}{1.2}
\begin{array}{c|c|c|c}
\widetilde{\mathcal{L}}_k(G_4) & S^i_k & p(S^i_k) & |\mathcal{O}_{p(S^i_k)}|\\[0.5\smallskipamount]
\hline
\hline
\widetilde{\mathcal{L}}_0(G_4) & S^1_0=\emptyset & (0,0,0,0) & 2\\
\hline 
\widetilde{\mathcal{L}}_1(G_4) & S^1_1=S^1_0\cup\{(0,0,0,1)\} & (0,0,0,1) & 8\\
\hline
\widetilde{\mathcal{L}}_2(G_4) & S^1_2=S^1_1\cup\{(0,0,1,1)\} & (0,0,1,2) & 24\\
\hline 
\widetilde{\mathcal{L}}_3(G_4) & S^1_3=S^1_2\cup\{(0,0,1,0)\} & (0,0,2,2) & 12\\[-\smallskipamount]  
& S^2_3=S^1_2\cup\{(0,1,0,1)\} & (0,1,1,3) & 24\\
\hline
\widetilde{\mathcal{L}}_4(G_4) & S^1_4=S^1_3\cup\{(0,1,1,1)\} & (0,1,3,3) & 24\\[-\smallskipamount]  
& S^2_4=S^2_3\cup\{(0,1,1,1)\} & (0,2,2,4) & 24\\[-\smallskipamount]  
& S^3_4=S^2_3\cup\{(1,0,0,1)\} & (1,1,1,4) & 8\\
\hline 
\widetilde{\mathcal{L}}_5(G_4) & S^1_5=S^1_4\cup\{(0,1,0,1)\} & (0,2,3,4) & 48\\[-\smallskipamount]  
& S^2_5=S^1_4\cup\{(1,0,1,1)\} & (1,1,4,4) & 12\\[-\smallskipamount]  
& S^3_5=S^2_4\cup\{(1,0,0,1)\} & (1,2,2,5) & 24\\
\hline
\widetilde{\mathcal{L}}_6(G_4) & S^1_6=S^1_5\cup\{(0,1,1,0)\} & (0,3,4,4) & 24\\[-\smallskipamount]  
& S^2_6=S^1_5\cup\{(1,0,1,1)\} & (1,2,4,5) & 48\\[-\smallskipamount]  
& S^3_6=S^3_5\cup\{(1,0,1,1)\} & (2,2,3,6) & 24\\
\hline
\widetilde{\mathcal{L}}_7(G_4) & S^1_7=S^1_6\cup\{(0,1,0,0)\} & (0,4,4,4) & 8\\[-\smallskipamount]  
& S^2_7=S^1_6\cup\{(1,0,1,1)\} & (1,3,5,5) & 24\\[-\smallskipamount]
& S^3_7=S^2_6\cup\{(1,0,0,1)\} & (2,2,4,6) & 24\\[-\smallskipamount]
& S^4_7=S^3_6\cup\{(1,1,0,1)\} & (3,3,3,7) & 8\\
\hline
\hline
& & & a(4)=\sum |\mathcal{O}_{p(S^i_k)}|=370\\
\end{array}
$$
\caption{Sizing the $4$-dimensional White Whale}\label{a4-vertices}
\end{table}

It is convenient to identify a generator $g$ with its binary representation.  For example, the generator
$$
g^j=(0,\dots,0,1,0,\dots,0,1)
$$
is identified with the integer $2^j+1$.

Likewise, the set $U_d^2$ of the generators summing up to the vertex
$$
p(U_d^2)=(1,\dots,1,d)
$$
that we considered in Proposition~\ref{sommet} can be identified with the set
$$
\{1,3,5\dots,2^{d-2}+1,2^{d-1}+1\}
$$
and the set $W_d^k$ of the generators  summing up to the vertex
$$
p(W_d^k)=(0,\dots,0,2^{k-1},\dots,2^{k-1}\}
$$
considered in item $(i)$ of Proposition~\ref{sommets} can be identified with the set
$$
\{1,2,3,\dots,2^{k}-1\}\mbox{.}
$$

Since the generation of the canonical vertices of $H_{\infty}^+(8,1)$ gives the vertices of $\widetilde{\mathcal{L}}_{k}(G_d)$ up to $k=8$ for all $d$, we can slightly warm-start Algorithm~\ref{LOOGd} by beginning the computation from $\widetilde{\mathcal{L}}_{8}(G_9)$. 

It might be quite speculative to draw any empirical intuition based on data available only till $d=9$. However, the following pattern may hold at least for the first $d$: the algorithm reaches relatively quickly the layer $\widetilde{\mathcal{L}}_{2^{d-2}+d}(G_d)$, the last $d$ layers are also relatively easy to compute, and the bulk of the computation results from the determination of the remaining $2^{d-2}-2d$ layers. Over this range, the size of the layers grows almost linearly to reach about $4\%$ of $a(d)$ for $d=7$, $2\%$ for $d=8$, and $1\%$ for $d=9$. Assuming that the same trend continues for $d=10$, Algorithm~\ref{LOOGd} would require the determination of a layer of size $0.5\%$ of $a(10)$ which is currently intractable as the determination of the largest layer of $a(9)$ already requires between one and two days.

\section{The edges of the White Whale}\label{edge-gen}

Consider a subset $S$ of $G_d$ and an element $g$ of $S$. Assume that both $p(S)$ and $p(S\mathord{\setminus}\{g\})$ are vertices of $H_\infty^+(d,1)$. Since $H_\infty^+(d,1)$ is zonotope, it must then have an edge with vertices $p(S)$ and $p(S\backslash \{g\})$. In other words, any edge of  $H_\infty^+(d,1)$ connects a vertex in $\mathcal{L}_{k-1}(G_d)$  to a vertex in $\mathcal{L}_{k}(G_d)$ for some $k$. As the proposed algorithms traverse the edges between two consecutive layers to generate the vertices, these algorithms can be used to generate the edges as well. However, in practice the number of edges can be significantly larger than the number of vertices and thus generating the edges of the White Whale quickly becomes intractable memory-wise. Consequently we propose an approach that, assuming that the vertices are determined by Algorithm~\ref{LOOGd}, counts the number of edges between $\mathcal{L}_{k-1}(G_d)$ and $\mathcal{L}_{k}(G_d)$ instead of generating them. The total number of edges is then obtained as a sum over $k$.

Given a vertex $p(S)$ of $H_\infty^+(d,1)$ distinct from the origin $p(\emptyset)$, let  $\delta^-_S$ denote the number of edges between $p(S)$ and a vertex in $\mathcal{L}_{|S|-1}(G_d)$:
$$
\delta^-_S=|\{g\in{S}:  p(S\backslash \{g\})\in\mathcal{L}_{|S|-1}(G_d)\}|\mbox{.}
$$

We also set $\delta^-_\emptyset=0$. The quantity $\delta^-_S$ can be seen as the {\em degree from below} of $p(S)$; that is, the number of edges between $p(S)$ and a vertex in the layer immediately below the one containing $p(S)$. Consider for example
$$
S=\{(0,0,1),(0,1,0),(0,1,1)\}\mbox{.}
$$

In that case, $p(S)$ is equal to $(0,2,2)$ and is indeed a vertex of $H_\infty^+(3,1)$. In fact, $p(S)$ is a vertex of the hexagonal facet of $H_\infty^+(3,1)$ contained in the hyperplane of equation $x_1=0$. In particular, both $p(S\backslash \{(0,0,1)\})$ and $p(S\backslash \{(0,1,0)\})$
are vertices of $H_\infty^+(3,1)$ while $p(S\backslash \{(0,1,1)\})$ is not. Thus $\delta^-_S=2$ as illustrated in Figure~\ref{Fig_H3-L}. By Proposition~\ref{degree-}, the degree from below of a vertex $p(S)$ of $H_\infty^+(d,1)$ is always $1$ when $S$ contains exactly $2^{d-1}$ generators.

\begin{prop}\label{degree-}
If $S$ contains exactly $2^{d-1}$ generators and $p(S)$ is a vertex of $H_\infty^+(d,1)$, then $\delta^-_S=1$. Moreover, exactly $|\mathcal{L}_{2^{d-1}}(G_d)|$ edges of the White Whale are equal to ${\bf 1}$ up to translation.
\begin{proof}
By Lemma~\ref{111} the vector ${\bf 1}$ belongs to $S$. According to the same proposition, $p(S\backslash\{g\})$ is not a vertex of $H_\infty^+(d,1)$ when $g$ is an element of $S$ other than ${\bf 1}$. Thus, $\delta^-_S = 1$ and the set of edges between $\mathcal{L}_{2^{d-1}-1}(G_d)$  and $\mathcal{L}_{2^{d-1}}(G_d)$ consists of exactly  $|\mathcal{L}_{2^{d-1}}(G_d)|$ edges equal, up to translation, to ${\bf 1}$, see Lemma~\ref{edge111}. As a consequence, $|\mathcal{L}_{2^{d-1}-1}(G_d)|=|\mathcal{L}_{2^{d-1}}(G_d)|$.
\end{proof}
\end{prop}

Summing up the edges encountered while traversing all the layers of  $H_\infty^+(d,1)$ yields that the number $e(d)$ of edges of the White Whale satisfies:
$$
e(d) 
=\sum_{k=1}^{2^d-1}   
 \sum_{p(S)\in\mathcal{L}_{k}(G_d)}
 \delta^-_S\mbox{.}
$$

\begin{table}
$$
\renewcommand{\arraystretch}{1.2}
\begin{array}{c|c|c|c|c|c}
\widetilde{\mathcal{L}}_k(G_3) & S^i_k & p(S^i_k) & |\mathcal{O}_{p(S^i_k)}| & \delta^-_{S^i_k} & |\mathcal{O}_{p(S^i_k)}|\delta^-_{S^i_k}\\[\smallskipamount]
\hline
\hline
\widetilde{\mathcal{L}}_1(G_3) & S^1_1=S^1_0\cup\{(0,0,1)\} & (0,0,1) & 6 & 1 & 6\\
\hline
\widetilde{\mathcal{L}}_2(G_3) & S^1_2=S^1_1\cup\{(0,1,1)\} & (0,1,2) & 12 & 1 & 12\\
\hline 
\widetilde{\mathcal{L}}_3(G_3) & S^1_3=S^1_2\cup\{(0,1,0)\} & (0,2,2) & 6 & 2 & 12\\[-\smallskipamount]  
& S^2_3=S^1_2\cup\{(1,0,1)\} & (1,1,3) & 6 & 2 & 12\\
\hline
\hline 
& & & & & e(3)=48\\
\end{array}
$$
\caption{Counting the edges of the $3$-dimensional White Whale}\label{a3-edges}
\end{table}
\begin{table}
$$
\renewcommand{\arraystretch}{1.2}
\begin{array}{c|c|c|c|c|c}
\widetilde{\mathcal{L}}_k(G_4) & S^i_k & p(S^i_k) & |\mathcal{O}_{p(S^i_k)}| & \delta^-_{S^i_k}  & |\mathcal{O}_{p(S^i_k)}|\delta^-_{S^i_k}\\[\smallskipamount]
\hline
\hline
\widetilde{\mathcal{L}}_1(G_4) & S^1_1=S^1_0\cup\{(0,0,0,1)\} & (0,0,0,1) & 8 & 1  & 8\\
\hline
\widetilde{\mathcal{L}}_2(G_4) & S^1_2=S^1_1\cup\{(0,0,1,1)\} & (0,0,1,2) & 24 & 1 & 24\\
\hline 
\widetilde{\mathcal{L}}_3(G_4) 	& S^1_3=S^1_2\cup\{(0,0,1,0)\} & (0,0,2,2) & 12 & 2 & 24\\[-\smallskipamount]
& S^2_3=S^1_2\cup\{(0,1,0,1)\} & (0,1,1,3) & 24 & 2 & 48\\
\hline
\widetilde{\mathcal{L}}_4(G_4) & S^1_4=S^1_3\cup\{(0,1,1,1)\} & (0,1,3,3) & 24 & 1 & 24\\[-\smallskipamount]  
& S^2_4=S^2_3\cup\{(0,1,1,1)\} & (0,2,2,4) & 24 & 1 & 24\\[-\smallskipamount]  
& S^3_4=S^2_3\cup\{(1,0,0,1)\} & (1,1,1,4) & 8 & 3 & 24\\
\hline 
\widetilde{\mathcal{L}}_5(G_4) 	& S^1_5=S^1_4\cup\{(0,1,0,1)\} 
& (0,2,3,4) & 48 & 2 & 96\\[-\smallskipamount]  
& S^2_5=S^1_4\cup\{(1,0,1,1)\} & (1,1,4,4) & 12 & 2 & 24\\[-\smallskipamount]  
& S^3_5=S^2_4\cup\{(1,0,0,1)\} 
& (1,2,2,5) & 24 & 2 & 48\\
\hline
\widetilde{\mathcal{L}}_6(G_4) 	& S^1_6=S^1_5\cup\{(0,1,1,0)\} & (0,3,4,4) & 24 & 2 & 48\\[-\smallskipamount]  
& S^2_6=S^1_5\cup\{(1,0,1,1)\} 
& (1,2,4,5) & 48 & 2 & 96\\[-\smallskipamount]  
& S^3_6=S^3_5\cup\{(1,0,1,1)\} & (2,2,3,6) & 24 & 2 & 48\\
\hline
\widetilde{\mathcal{L}}_7(G_4) 	& S^1_7=S^1_6\cup\{(0,1,0,0)\} & (0,4,4,4) & 8 & 3 & 24\\[-\smallskipamount]  
& S^2_7=S^1_6\cup\{(1,0,1,1)\} 
& (1,3,5,5) & 24 & 3 & 72\\[-\smallskipamount]
& S^3_7=S^2_6\cup\{(1,0,0,1)\} 
& (2,2,4,6) & 24 & 3 & 72\\[-\smallskipamount]
& S^4_7=S^3_6\cup\{(1,1,0,1)\} & (3,3,3,7) & 8 & 3 & 24\\
\hline
\hline
& & & & & e(4)=760\\
\end{array}
$$
\caption{Counting the edges of the $4$-dimensional White Whale}\label{a4-edges}
\end{table}

The White Whale being centrally symmetric, the summation can be done up to $k=2^{d-1}-1$ to account for all the edges except for the $|\mathcal{L}_{2^{d-1}}(G_d)|$ edges between $\mathcal{L}_{2^{d-1}-1}(G_d)$  and $\mathcal{L}_{2^{d-1}}(G_d)$ identified in Proposition~\ref{degree-}. Further exploiting the symmetry group of $H_\infty^+(d,1)$, we obtain 
$$
e(d)
= 
\left(
\sum_{k=1}^{2^{d-1}-1}   
 \sum_{p(S)\in\widetilde{\mathcal{L}}_{k}(G_d)}
|\mathcal{O}_{p(S)}|   \:  \delta^-_S 
\right)
+
\left(
\sum_{p(S)\in\widetilde{\mathcal{L}}_{2^{d-1}-1}(G_d)} \frac{|\mathcal{O}_{p(S)}|}{2}
\right)
$$
where $|\mathcal{O}_{p(S)}|$ denotes the size of the orbit generated by the action of the symmetry group of $H_{\infty}^+(d,1)$ on a canonical vertex $p(S)$. By this calculation, illustrated in Table~\ref{a3-edges}, the $3$-dimensional White Whale has
$$
(6\times 1+12\times 1+6\times 2 +6\times 2)+\left(\frac{6}{2}+\frac{6}{2}\right)=48
$$
edges, see Figure~\ref{Fig_H3-L}. The corresponding calculation, but in the case of the $4$\nobreakdash-dimensional White Whale is illustrated in Table~\ref{a4-edges}.

The values of $e(d)$ are yielded by two rounds of calculation, which are based on the output of $a(d)$ and deployed as \texttt{Spark} two sets of pipelines. The first set of pipelines are focused on the connectivity between consecutive layers, whose output is further passed to another set of pipelines to produce degree reports of each layer. The resulting number of edges are reported in Table \ref{final}.

\section{The vertex degrees of the White Whale}\label{sec:degree}

Similarly to the degree from below defined in Section~\ref{edge-gen}, we denote by $\delta^+_S$ the {\em degree from above} of a vertex $p(S)$ distinct from $p(G_d)$; that is, the number of edges connecting $p(S)$ to a vertex contained in the layer $\mathcal{L}_{|S|+1}(G_d)$.
$$
\delta^+_S=|\{g\notin{S}:  p(S\cup \{g\})\in\mathcal{L}_{|S|+1}(G_d)\}|\mbox{.}
$$

In addition, we set $\delta^+_{G_d}$ to $0$.
As $H_\infty^+(d,1)$ is centrally symmetric, Proposition~\ref{degree-} can be rewritten as follows. 

\begin{prop}\label{degree+}
If a subset $S$ of $G_d$ contains exactly $2^{d-1}-1$ generators and $p(S)$ is a vertex of $H_\infty^+(d,1)$, then $\delta^+_S=1$. 
\end{prop}
The degree $\delta_S$ of a vertex $p(S)$; that is, the number of edges of $H_\infty^+(d,1)$ incident to $p(S)$, is given by $\delta_S=\delta^-_S+\delta^+_S$.  Note that $\delta_{\emptyset}$ and $\delta_{G_d}$ are both equal to $d$. For example, the $32$ vertices of $H_\infty^+(3,1)$ are all of degree $3$. in other words, $H_\infty^+(3,1)$ is a simple zonotope, see Table~\ref{a3-edges-} and Figure~\ref{Fig_H3-L}.

\begin{table}
$$
\renewcommand{\arraystretch}{1.2}
\begin{array}{c|c|c|c|c||c}
\widetilde{\mathcal{L}}_k(G_3) & S^i_k & p(S^i_k) & \delta^-_{S^i_k}  & \delta^+_{S^i_k} & \delta_{S^i_k}\\[\smallskipamount]
\hline
\hline
\widetilde{\mathcal{L}}_0(G_3) & S^1_0=\emptyset & (0,0,0) & 0 & 3 & 3\\
\hline 
\widetilde{\mathcal{L}}_1(G_3) & S^1_1=S^1_0\cup\{(0,0,1)\} & (0,0,1) & 1 & 2 & 3\\
\hline
\widetilde{\mathcal{L}}_2(G_3) & S^1_2=S^1_1\cup\{(0,1,1)\} & (0,1,2) & 1 & 2 & 3\\
\hline 
\widetilde{\mathcal{L}}_3(G_3) & S^1_3=S^1_2\cup\{(0,1,0)\} & (0,2,2) & 2 & 1 & 3\\[-\smallskipamount]
& S^2_3=S^1_2\cup\{(1,0,1)\} & (1,1,3) & 2 & 1 & 3\\
\end{array}
$$
\caption{The vertex degrees of the $3$-dimensional White Whale}\label{a3-edges-}
\end{table}
\begin{table}
$$
\renewcommand{\arraystretch}{1.2}
\begin{array}{c|c|c|c|c||c}
\widetilde{\mathcal{L}}_k(G_4) & S^i_k & p(S^i_k)  & \delta^-_{S^i_k}  & \delta^+_{S^i_k} & \delta_{S^i_k}\\[\smallskipamount]
\hline
\hline
\widetilde{\mathcal{L}}_0(G_4) & S^1_0=\emptyset & (0,0,0,0) & 0 & 4 & 4\\
\hline
\widetilde{\mathcal{L}}_1(G_4) & S^1_1=S^1_0\cup\{(0,0,0,1)\} & (0,0,0,1) & 1 & 3 & 4\\
\hline
\widetilde{\mathcal{L}}_2(G_4) & S^1_2=S^1_1\cup\{(0,0,1,1)\} & (0,0,1,2) & 1 & 3 & 4\\
\hline 
\widetilde{\mathcal{L}}_3(G_4) & S^1_3=S^1_2\cup\{(0,0,1,0)\} & (0,0,2,2) & 2 & 2 & 4\\[-\smallskipamount]
& S^2_3=S^1_2\cup\{(0,1,0,1)\} & (0,1,1,3) & 2 & 2 & 4\\
\hline
\widetilde{\mathcal{L}}_4(G_4) & S^1_4=S^1_3\cup\{(0,1,1,1)\} & (0,1,3,3) & 1 & 3 & 4\\[-\smallskipamount]
& S^2_4=S^2_3\cup\{(0,1,1,1)\} & (0,2,2,4) & 1 & 3 & 4\\[-\smallskipamount]
& S^3_4=S^2_3\cup\{(1,0,0,1)\} & (1,1,1,4) & 3 & 3 & 6\\
\hline 
\widetilde{\mathcal{L}}_5(G_4) & S^1_5=S^1_4\cup\{(0,1,0,1)\} 
& (0,2,3,4) & 2 & 2 & 4\\[-\smallskipamount]
& S^2_5=S^1_4\cup\{(1,0,1,1)\} & (1,1,4,4) & 2 & 4 & 6\\[-\smallskipamount]
& S^3_5=S^2_4\cup\{(1,0,0,1)\} 
& (1,2,2,5) & 2 & 2 & 4\\
\hline
\widetilde{\mathcal{L}}_6(G_4) & S^1_6=S^1_5\cup\{(0,1,1,0)\} & (0,3,4,4) & 2 & 2 & 4\\[-\smallskipamount]
& S^2_6=S^1_5\cup\{(1,0,1,1)\} 
& (1,2,4,5) & 2 & 2 & 4\\[-\smallskipamount]
& S^3_6=S^3_5\cup\{(1,0,1,1)\} & (2,2,3,6) & 2 & 2 & 4\\
\hline
\widetilde{\mathcal{L}}_7(G_4) & S^1_7=S^1_6\cup\{(0,1,0,0)\} & (0,4,4,4) & 3 & 1 & 4\\[-\smallskipamount]
& S^2_7=S^1_6\cup\{(1,0,1,1)\} 
& (1,3,5,5) & 3 & 1 & 4\\[-\smallskipamount]
& S^3_7=S^2_6\cup\{(1,0,0,1)\} 
& (2,2,4,6) & 3 & 1 & 4\\[-\smallskipamount]
& S^4_7=S^3_6\cup\{(1,1,0,1)\} & (3,3,3,7) & 3 & 1 & 4\\
\end{array}
$$
\caption{The vertex degrees of the $4$-dimensional White Whale}\label{a4-edges-}
\end{table}
The calculation of the vertex-degrees of the $4$-dimensional White Whale is illustrated in Table~\ref{a4-edges-}. The number $o(d)$ of orbits or, equivalently the number of canonical vertices, the average vertex degree $2e(d)/a(d)$, and the average size of an orbit $a(d)/o(d)$ are all given up to dimension $9$ in Table~\ref{final}. These initial values may indicate that the average size of an orbit $a(d)/o(d)$ is a large fraction of the largest possible orbit size of $2d!$.

\begin{table}[b]
\makebox[\linewidth]{
$
\begin{array}{c|c|c|c|c|c}

d & a(d) & e(d)  &  \frac{2e(d)}{a(d)} & o(d) & \frac{a(d)}{2d!o(d)} \\
\hline
2 & 6 & 6 & 2  & 2 & 75\%\\
3 & 32 & 48 & 3 & 5 & \approx 53\%\\
4 & 370 & 760 & \approx 4.1 & 18 & \approx 43\%\\ 
5 & 11\,292 & 30\,540 & \approx 5.4  & 112  &  \approx 43\%\\ 
6 & 1\,066\,044 & 3\,662\,064 & \approx 6.9 & 1\:512  & \approx 49\%\\  
7 & 347\,326\,352 & 1\,463\,047\,264 & \approx 8.4 & 56\:220 & \approx 61\%\\ 
8 & 419\,172\,756\,930 & 2\,105\,325\,742\,608 & \approx 10.0 &  6\:942\:047 & \approx 75\%\\
9 & 1\,955\,230\,985\,997\,140 & 11\,463\,171\,860\,268\,180 & \approx 11.7  & 3\,140\,607\,258 & \approx 86\%
\end{array}
$
}
\smallskip

\caption{Some sizes of the White Whale.}\label{final}
\end{table}

\begin{rem}
All the known values of $e(d)$ are multiples of $d(d+1)$ and, when $d$ is equal to $7$, we obtain from Table~\ref{final} that

$$
\frac{e(d)}{4d(d+1)}= 6\,531\,461\mbox{,}
$$
which is a  prime number.
\end{rem}


Let us now turn our attention back to the vertices $p(U_d^k)$ of $H_\infty^+(d,1)$ provided by Proposition~\ref{sommet}. We can determine exactly the degree of these vertices.

\begin{lem}\label{expo0}
The degree of $p(U_d^k)$ from below is $\displaystyle\delta^-_{U_d^k}=\displaystyle{d-1 \choose k-1}$.

\end{lem}
\begin{proof}
We recall that $U_d^k$ is defined when $1\leq{k}\leq{d-1}$.
Let us first show that, if $g$ belongs to $U_d^k\mathord{\setminus}U_d^{k-1}$, then $p(U_d^k\mathord{\setminus}\{g\})$ is a vertex of $H_\infty^+(d,1)$. Observe that, when $k=1$, this is immediate as the origin of $\mathbb{R}^d$ is a vertex of $H_\infty^+(d,1)$. Hence we can assume that $k\geq2$. By symmetry, we can moreover assume without loss of generality that $g$ is the generator whose last $k$ coordinates are equal to $1$ and whose first $d-k$ coordinates are equal to $0$. We will use the linear optimization oracle ($LO_{S,G}$) with $S=U_d^k\mathord{\setminus}\{g\}$ and $G=G_d$. 

Consider the vector $c$ of $\mathbb{R}^d$ whose first $d-k$ coordinates are equal to $2-3k$, whose last coordinate is $3k^2-3k-1$, and whose remaining $d-k-1$ coordinates are $-3k$. Consider a vector $g'$ in $U_d^k\mathord{\setminus}\{g\}$. As $g'$ is distinct from $g$, either at least one of its $d-k$ first coordinates is non-zero, and
$$
\sum_{i=1}^{d-1}c_ig'_i\geq (2-3k)-3k(k-2)=-3k^2+3k+2\mbox{,}
$$
or at most $k-2$ of its $d-1$ first coordinates are non-zero, and
$$
\sum_{i=1}^{d-1}c_ig'_i\geq -3k(k-2)=-3k^2+6k\mbox{.}
$$

As $c_d=3k^2-3k-1$ and $k\geq1$, both of these inequalities imply that $c^Tg'\geq1$. Now consider a vector $g'$ in $G_d\mathord{\setminus}[U_d^k\mathord{\setminus}\{g\}]$. If $g'_d=0$, then $c^Tg'\leq-1$ because $g'$ has at least one non-zero coordinate and the first $d-1$ coordinates of $c$ are negative. If $g'_d=1$, then either $g'=g$ or at least $k$ of its $d-1$ first coordinates are non-zero. If $g'=g$, then by construction, 
$$
c^Tg'=-3k(k-1)+3k^2-3k-1=-1\mbox{.}
$$

If at least $k$ of the $d-1$ first coordinates of $g'$ are non-zero, then
$$
c^Tg'\leq(2-3k)k+3k^2-3k-1<-1\mbox{.}
$$

This proves that $p(U_d^k\mathord{\setminus}\{g\})$ is a vertex of $H_\infty^+(d,1)$, as desired.

We now show that, if $g$ belongs to $U_d^{k-1}$, then $p(U_d^k\mathord{\setminus}\{g\})$ is not a vertex of $H_\infty^+(d,1)$. As $U_d^k\mathord{\setminus}U_d^{k-1}$ contains exactly
$$
{d-1 \choose k-1}
$$
vectors, this will prove the proposition. Consider a vector $g$ from $U_d^{k-1}$. By symmetry, we can assume without loss of generality that the last $k-1$ coordinates of $g$ are equal to $1$ and that its first $d-k+1$ coordinates are equal to $0$. Denote by $g'$ the vector in $U_d^k$ whose $k$ last coordinates are equal to $1$ and by $g''$ the vector in $G_d\mathord{\setminus}U_d^k$ whose unique non-zero coordinate is $g''_{d-k+1}$.

By construction, $g=g'-g''$ and as an immediate consequence,
$$
p(U_d^k\mathord{\setminus}\{g\})=p([U_d^k\mathord{\setminus}\{g'\}]\cup\{g''\})\mbox{.}
$$

This proves that $p(U_d^k\mathord{\setminus}\{g\})$ can be decomposed as a sum of two different subsets of $G_d$. Therefore, this point cannot be a vertex of $H_\infty^+(d,1)$.
\end{proof}

\begin{lem}\label{expo1}
The degree of $p(U_d^k)$ from above is $\displaystyle\delta^+_{U_d^k}=\displaystyle{d-1 \choose k}$.
\end{lem}
\begin{proof}
We recall that $U_d^k$ is defined when $1\leq{k}\leq{d-1}$. The proof proceeds as that of Lemma~\ref{expo0}. Consider a vector $g$ that belongs to $U_d^{k+1}\mathord{\setminus}U_d^k$. We show as a first step that $p(U_d^k\cup\{g\})$ is a vertex of $H_\infty^+(d,1)$ by using the oracle $(LO_{S,G})$ with $S=U_d^k\cup\{g\}$ and $G=G_d$.

By symmetry, we can assume without loss of generality that the last $k+1$ coordinates of $g$ are non-zero. Consider the vector $c$ of $\mathbb{R}^d$ whose first $d-k-1$ coordinates are equal to $-2k-1$, whose last coordinate is equal to $2k^2-k+1$ and whose other $k$ coordinates are equal to $-2k+1$. Further consider a vector $g'$ in $U_d^k\cup\{g\}$. If $g'$ is equal to $g$, then by construction
$$
c^Tg'=k(-2k+1)+2k^2-k+1=1\mbox{.}
$$

If $g'$ is not equal to $g$, then at most $k-1$ of its first $d-1$ coordinates are non-zero. As a consequence,
$$
\sum_{i=1}^{d-1}c_ig'_i\geq-(k-1)(2k+1)=-2k^2+k+1\mbox{.}
$$

As $c_d=2k^2-k+1$ and $g'_d=1$, this yields $c^Tg'\geq2$. So far, we have shown that $c^Tg'\geq1$ for every $g'$ in $U_d^k\cup\{g\}$. Now let us consider a vector $g'$ in $G_d\mathord{\setminus}[U_d^k\cup\{g\}]$ and show that $c^Tg'\leq-1$. If $g'_d=0$, then $c^Tg'$ must be negative because $g'$ has at least one non-zero coordinate and the $d-1$ first coordinates of $c$ are negative. If $g'_d=1$, then $g'$ must have at least $k+1$ non-zero coordinates. As in addition $g'$ is distinct from $g$, at least one its first $d-k-1$ coordinates is equal to $1$. As a consequence,
$$
\sum_{i=1}^{d-1}c_ig'_i\leq-(2k+1)-(k-1)(2k-1)=-2k^2+k-2\mbox{.}
$$

Since $c_d=2k^2-k+1$ and $g'_d=1$, this yields $c^Tg'\leq-1$. According to the oracle $(LO_{S,G})$ with $S=U_d^k\cup\{g\}$ and $G=G_d$, the point $p(U_d^k\cup\{g\})$ is then necessarily a vertex of $H_\infty^+(d,1)$, as desired.

Let us now show that for any vector $g$ in $G_d\mathord{\setminus}U_d^{k+1}$, the point $p(U_d^k\cup\{g\})$ is never a vertex of $H_\infty^+(d,1)$.  Denote by $j$ the number of non-zero coordinates of $g$ and assume, first that $g_d=0$. By symmetry, we can further assume without loss of generality that $g_i=1$ exactly when $d-j\leq{i}\leq{d-1}$. Denote by $g'$ the vector in $G_d\mathord{\setminus}U_d^k$ such that $g'_i=1$ when
$$
d-\max\{j,k\}\leq{i}\leq{d}\mbox{.}
$$

By construction, $g'-g$ belongs to $U_d^k$ but $g'$ does not. Moreover,
$$
p(U_d^k\cup\{g\})=p([U_d^k\mathord{\setminus}\{g'-g\}]\cup\{g'\})\mbox{.}
$$

This shows that $p(U_d^k\cup\{g\})$ admits two decompositions into a sum of vectors from $G_d$ and therefore cannot be a vertex of $H_\infty^+(d,1)$.

Finally, assume that $g_d=1$. In this case, $j$ is at least $k+2$. By symmetry we can further assume that last $j$ coordinates of $g$ are non-zero. Denote by $g'$ the vector in $G_d$ whose only non-zero coordinate is $c_{d-1}$ and observe that $g-g'$ does not belong to $U_d^k$ because it has at least $k+1$ non-zero coordinates. Moreover, $g'$ does not belong to $U_d^k\cup\{g\}$ either, and
$$
p(U_d^k\cup\{g\})=p(U_d^k\cup\{g-g',g'\})\mbox{.}
$$

As above, this shows that $p(U_d^k\cup\{g\})$ admits two decompositions into a sum of vectors from $G_d$. Therefore, it cannot be a vertex of $H_\infty^+(d,1)$.

As there are exactly
$$
{d-1 \choose k}
$$
vectors in $U_d^{k+1}\mathord{\setminus}U_d^k$, this proves the lemma.
\end{proof}

\begin{thm}\label{expo}
The degree of $p(U_d^k)$ is $\displaystyle{d \choose k}$.
\end{thm}
\begin{proof}
Theorem~\ref{expo} immediately follows from Lemmas~\ref{expo0} and~\ref{expo1}.
\end{proof}

\begin{cor} Setting $k=\lfloor d/2 \rfloor$ in  Theorem~\ref{expo} yields a vertex of degree
$$
{d \choose \lfloor d/2 \rfloor }\geq\frac{2^d}{d+1}
$$
thus providing the announced vertex of the White Whale whose degree is exponential in the dimension.
\end{cor}

\begin{rem}
The degree of $p(U_d^k)$ is maximal among the vertices of $H_\infty^+(d,1)$ when $2\leq{d}\leq9$. We hypothesize it is the case in any dimension. 
\end{rem}

\noindent{\bf Acknowledgement.}  The authors thank the anonymous referees, Zachary Chroman, Lukas K\"uhne, and Mihir Singhal for providing valuable comments and suggestions.
\bibliography{WhiteWhale}
\bibliographystyle{ijmart}

\end{document}